\documentclass[12pt]{amsart}%
\usepackage{graphicx,amssymb,latexsym,amsfonts,amsmath,amsthm}
\usepackage{color}
\usepackage{amsmath}
\usepackage{amsfonts}
\usepackage{amssymb}
\usepackage{graphicx}
\usepackage{hyperref}%
\providecommand{\U}[1]{\protect\rule{.1in}{.1in}}
\providecommand{\U}[1]{\protect\rule{.1in}{.1in}}
\advance\hoffset by -0.9 truecm
\hypersetup{
colorlinks=true,           linkcolor=blue,              citecolor=blue,            filecolor=blue,          urlcolor=blue           }

\newtheorem{theo}{Theorem}
\newtheorem{coro}{Corollary}
\newtheorem{lemm}{Lemma}
\theoremstyle{remark}

\begin{document}
\title[On the connectedness of the set of Riemann surfaces...]{On the connectedness of the set of Riemann surfaces with real moduli}
\author{Antonio F. Costa}
\address{Departamento de Matem\'aticas Fundamentales, Facultad de Ciencias, UNED, 28040
Madrid, Spain}
\email{acosta@mat.uned.es}
\author{Rub\'en A. Hidalgo}
\address{Departamento de Matem\'atica y Estad\'{\i}stica, Universidad de La Frontera.
Temuco, Chile}
\email{ruben.hidalgo@ufrontera.cl}
\thanks{Partially supported by Project MTM2014-55812-P (Spanish Ministry of
Competitivity), Project of Fondecyt 1150003 and Project Anillo ACT1415 PIA CONICYT}

\begin{abstract}
The moduli space ${\mathcal{M}}_{g}$, of genus $g\geq2$ closed Riemann
surfaces, is a complex orbifold of dimension $3(g-1)$ which carries a natural
real structure i.e. it admits an anti-holomorphic involution $\sigma$. The
involution $\sigma$ maps each point corresponding to a Riemann surface $S$ to
its complex conjugate $\overline{S}$. The fixed point set of $\sigma$ consists
of the isomorphism classes of closed Riemann surfaces admitting an
anticonformal automorphism. Inside $\mathrm{Fix}(\sigma)$ is the locus
${\mathcal{M}}_{g}(\mathbb{R})$, the set of real Riemann surfaces, which is
known to be connected by results due to P. Buser, M. Sepp\"{a}l\"{a} and R.
Silhol. The complement $\mathrm{Fix}(\sigma)-{\mathcal{M}}_{g}(\mathbb{R})$
consists of the so called pseudo-real Riemann surfaces, which is known to be
non-connected. In this short note we provide a simple argument to observe that
$\mathrm{Fix}(\sigma)$ is connected.

\end{abstract}
\maketitle

%%%%%%%%%%%%%%%%%
%%%%%%%%%%%%%%%%%

\section{Introduction}

The moduli space ${\mathcal{M}}_{g}$, of genus $g\geq2$ closed Riemann
surfaces, is a complex orbifold of dimension $3(g-1)$. The study of this
moduli space was already started by F. Klein. This space carries a natural
real structure given by an involution $\sigma$ which sends each Riemann
surface to its complex conjugate. The fixed point set of $\sigma$ consists of
the isomorphism classes of closed Riemann surfaces admitting an anticonformal
automorphism and those surfaces are said to have real moduli. The quotient
space ${\mathcal{M}}_{g}/\left\langle \sigma\right\rangle $ is the moduli
space of Riemann surfaces of genus $g$ considered as Klein surfaces, i. e. two
surfaces are equivalent if they are holomorphic or anti-holomorphically
equivalent. Inside $\mathrm{Fix}(\sigma)$ is the locus ${\mathcal{M}}%
_{g}(\mathbb{R})$, consisting of those admitting an anticonformal involution
(Riemann surfaces corresponding to real algebraic curves), which is known to
be connected by results due to P. Buser, M. Sepp\"{a}l\"{a} and R. Silhol
\cite{BSS,Seppala} (see also a proof in \cite{CI}). The locus ${\mathcal{M}%
}_{g}(\mathbb{R})$ has been studied by Klein in \cite{K} and the fact that in
general $\mathrm{Fix}(\sigma)\neq{\mathcal{M}}_{g}(\mathbb{R})$ was noted by
C. Earle in \cite{E}. The complement ${\mathcal{P}}_{g}=\mathrm{Fix}%
(\sigma)-{\mathcal{M}}_{g}(\mathbb{R})$ consists of the so called pseudo-real
Riemann surfaces. In \cite{BCC} it was observed that ${\mathcal{P}}_{g}$ is
non-empty for every $g\geq2$ (then $\mathrm{Fix}(\sigma)\neq{\mathcal{M}}%
_{g}(\mathbb{R})$ for all $g\geq2$). In \cite{BC} it was observed that
${\mathcal{P}}_{2}$ and ${\mathcal{P}}_{3}$ are connected and ${\mathcal{P}%
}_{4}$ is non-connected (it has three connected components). There are
infinite many integers $n_{i}$ such that ${\mathcal{P}}_{n_{i}}$ is not
connected (this fact follows from Theorems 3.4, 5.4, 5.5 and Section 6 of
\cite{BCC} to construct families of pseudo-real surfaces and from the
Corollary of Theorem 2 of \cite{Si} to prove that these families are in
different connected components).

In this short note we provide a simple argument to observe that $\mathrm{Fix}%
(\sigma)$ is connected.

\begin{theo}
\label{conexo} The set $\mathrm{Fix}(\sigma)$ is connected.
\end{theo}

This connectedness fact may not be a surprise, but it seems it has not been
noted in the existent literature.

\textbf{Acknowledgements}. We wish to thank the referee for corrections and suggestions.

%%%%%%%%%%%%%%%%%%
%%%%%%%%%%%%%%%%%%

\section{The set of Riemann surfaces with real moduli is connected}

In this section, we proceed to prove Theorem \ref{conexo}. Let $S$ be a closed
Riemann surface of genus $g\geq2$, admitting an anticonformal automorphism
$\tau$ of order $2n$, where $n\geq2$ is even. The quotient orbifold
${\mathcal{O}}=S/\langle\tau\rangle$ is homeomorphic to the connected sum of
some $\gamma$ real projective planes and has exactly $r$ cone points, say of
orders $n_{1},\ldots,n_{r}\in\{2,\ldots,n\}$, where each $n_{j}$ is a divisor
of $n$. This means that there is an NEC group $\Delta$, acting on the unit
disc ${\mathbb{D}}$, with signature $(\gamma;-;[n_{1},\ldots,n_{r}])$ and
there is a surjective homomorphism $\rho:\Delta\rightarrow C_{2n}=\langle
\tau\rangle$, whose kernel $\Gamma$ is a Fuchsian group uniformizing $S$ and
$\langle\tau\rangle$ is induced by $\Delta$, that is, ${\mathbb{D}}%
/\Gamma=S\rightarrow{\mathbb{D}}/\Delta=S/\langle\tau\rangle$.

The locus ${\mathcal{O}}(S,\tau)$ in moduli space ${\mathcal{M}}_{g}$
consisting of those (classes of) closed Riemann surfaces $\widehat{S}$
admitting an anticonformal automorphism $\widehat{\tau}$ of order $2n$ so
there is an orientation preserving homeomorphism $\phi:S\rightarrow
\widehat{S}$ conjugating $\tau$ to $\widehat{\tau}$ is connected \cite{BC}.
Now, since the locus ${\mathcal{M}}_{g}(\mathbb{R})$ is connected, in order to
check the connectivity of $\mathrm{Fix}(\sigma)$, we only need to find a point
$[\widehat{S}]\in{\mathcal{O}}(S,\tau)$ so that $\widehat{S}$ admits also an
anticonformal involution. In the NEC group setting, this is equivalent to
finding an NEC group $K$ containing reflections (i.e. the group $K$
uniformizes a bordered Klein surface) and a subgroup $\widehat{\Delta}$ of $K$
so that there is an isomorphism $\iota:\Delta\rightarrow\widehat{\Delta}$ with
$\iota(\Gamma)$ being a normal subgroup of $K$. Note that $\iota(\Gamma)$ is a
Fuchsian group with $\widehat{S}={\mathbb{D}}/\iota(\Gamma)$ a closed Riemann
surface which has an automorphism $\widehat{\tau}$ topologically equivalent to
$\tau$ ($\widehat{\tau}$ is an automorphism group of the cyclic covering
${\mathbb{D}}/\iota(\Gamma)=\widehat{S}\rightarrow{\mathbb{D}}/\widehat{\Delta
}$) and $\widehat{S}$ has anticonformal involutions too, produced by the
lifting of the reflections of $K$ to $\widehat{S}$ (note that $\widehat{S}$
has empty boundary but ${\mathbb{D}}/K$ is bordered). In this way we have
${\mathcal{O}}(S,\tau)\cap{\mathcal{M}}_{g}(\mathbb{R})\neq\varnothing$ for
every pseudo-real surface $S $; this implies $\mathrm{Fix}(\sigma)$ is connected.

%%%%%%%%%%%%%%%%%%%

\subsection{The construction of $K$}

Let $K$ be an NEC group uniformizing the closed disc, with $\gamma$ interior
cone points of order $2$ and $r$ cone points in its border of orders
$n_{1},\ldots,n_{r}$, that is, an NEC\ group of signature
$(0;+;[2,\overset{\gamma}{...},2],\{(n_{1},...,n_{r})\})$. A canonical
presentation for $K$ is as follows
\[
K=\langle x_{1},\ldots,x_{\gamma},e,\tau_{1},\ldots,\tau_{r+1}:x_{1}%
^{2}=\cdots=x_{\gamma}^{2}=\tau_{1}^{2}=\cdots=\tau_{r+1}^{2}=1,
\]%
\[
e^{-1}\tau_{r+1}e=\tau_{1},x_{\gamma}\cdots x_{2}x_{1}e=1,(\tau_{1}\tau
_{2})^{n_{1}}=\cdots=(\tau_{r}\tau_{r+1})^{n_{r}}=1\rangle,
\]
where the elements $x_{j}$ are elliptic transformations of order two, the
elements $\tau_{j}$ are reflections and $e$ is an hyperbolic or elliptic
element (see \cite{BEGG} and Figure \ref{fig:FD1} for a fundamental domain of
$K$).

\begin{figure}[ptb]
\centering
\includegraphics[width=3.0cm]{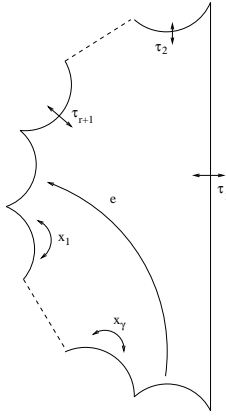}\caption{A fundamental domain for $K$}%
\label{fig:FD1}%
\end{figure}

%%%%%%%%%%%%%%%%%%%

\subsection{A subgroup $\protect\widehat{\Delta}$ of $K$}

Let us consider the surjective homomorphism
\[
\theta:K\rightarrow C_{2}=\langle a:a^{2}=1\rangle
\]%
\[
x_{1},\ldots,x_{\gamma},\tau_{1},\ldots,\tau_{r+1}\mapsto a,\;e\mapsto1.
\]

The kernel $\widehat{\Delta}$ of $\theta$ has no reflections and contains
orientation reversing elements (for instance $c_{1}x_{1}$); so its signature
must be of the form $(h;-;[m_{1},...,m_{s}])$, that is, ${\mathbb{D}%
}/\widehat{\Delta}$ is the connected sum of $h$ real projective planes and
contains exactly $s$ cone points, these having orders $m_{1},\ldots,m_{s}$.
Using the Riemann-Hurwitz formula and the usual methods to compute the
signature of an NEC\ subgroup (see \cite{BEGG}), we have that $h=\gamma$,
$s=r$ and $n_{j}=m_{j}$. So, $\Delta$ is isomorphic to $\widehat{\Delta}$; let
$\iota:\Delta\rightarrow\widehat{\Delta}$ be such an isomorphism. A
fundamental domain for $\widehat{\Delta}$ is shown in Figure \ref{fig:FD2},
this given as the union of the previous fundamental domain for $K$ with its
image under the reflection $\tau_{1}$. By the Poincar\'{e} polygon theorem (or
using the Schreier-Reidemeister method) a presentation of $\widehat{\Delta}$,
in terms of the generators of $K$, may be obtained. We have as generators
\[
\delta_{1}=\tau_{1}x_{1},\ldots,\delta_{\gamma}=\tau_{1}x_{\gamma},c_{1}%
=\tau_{1}\tau_{2},\ldots,c_{r}=\tau_{1}\tau_{r+1},e_{1}=e,e_{2}=\tau_{1}%
e\tau_{1}%
\]
satisfying the following relations
\[
c_{1}^{n_{1}}=1,
\]%
\[
(c_{1}^{-1}c_{2})^{n_{2}}=(c_{2}^{-1}c_{3})^{n_{3}}\cdots=(c_{r-1}^{-1}%
c_{r})^{n_{r}}=1,
\]%
\[
e_{1}e_{2}^{-1}c_{r}=1,
\]%
\[
\left.
\begin{array}
[c]{lll}%
\delta_{1}^{-1}\delta_{2}\delta_{3}^{-1}\cdots\delta_{\gamma-1}^{-1}%
\delta_{\gamma} & = & e_{1}\\
\delta_{1}\delta_{2}^{-1}\delta_{3}\cdots\delta_{\gamma-1}\delta_{\gamma}^{-1}
& = & e_{2}%
\end{array}
\right\}  \;(\mbox{if $\gamma$ is even})
\]%
\[
\left.
\begin{array}
[c]{lll}%
\delta_{1}\delta_{2}^{-1}\delta_{3}\cdots\delta_{\gamma-1}^{-1}\delta_{\gamma}
& = & e_{1}\\
\delta_{1}^{-1}\delta_{2}\delta_{3}^{-1}\cdots\delta_{\gamma-1}\delta_{\gamma
}^{-1} & = & e_{2}%
\end{array}
\right\}  \;(\mbox{if $\gamma$ is odd})
\]

\begin{figure}[ptb]
\centering
\includegraphics[width=6.0cm]{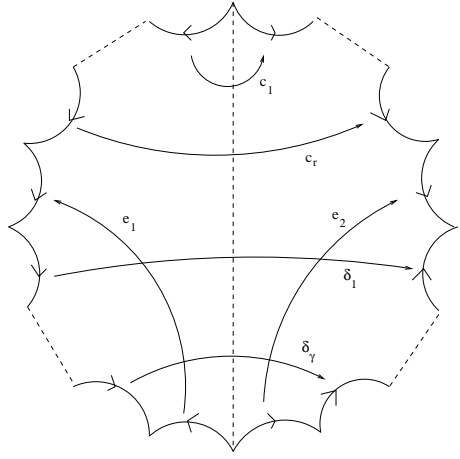}\caption{A fundamental domain for
$\protect\widehat{\Delta}$}%
\label{fig:FD2}%
\end{figure}

%%%%%%%%%%%%%%%%%%%

\subsection{The final step}

Let us consider the surjective homomorphism $\eta=\rho\circ\iota
^{-1}:\widehat{\Delta} \to C_{2n}$, whose kernel is $\widehat{\Gamma}%
=\iota(\Gamma) $; a torsion free Fuchsian group that uniformizes a closed
Riemann surface ${\widehat{S}}$. In order to finish our proof, we only need to
check that $\widehat{\Gamma}$ is a normal subgroup of $K$. This is what the
following general lemma asserts.

\begin{lemm}
\label{lema1} Let $A$ be an abelian group and let $\zeta:\widehat{\Delta} \to
A$ be a homomorphism. Then $e_{1}e_{2} \in\ker(\zeta)$ and $\ker(\zeta) \lhd K
$.
\end{lemm}

\begin{proof}
We assume $\gamma$ even; the odd case is similar. The relations
\[
\delta_{1}^{-1}\delta_{2}\delta_{3}^{-1}\cdots\delta_{\gamma-1}^{-1}%
\delta_{\gamma}=e_{1},\;\mbox{ and }\;\delta_{1}\delta_{2}^{-1}\delta
_{3}\cdots\delta_{\gamma-1}\delta_{\gamma}^{-1}=e_{2},
\]
assert that $\zeta(e_{1})=\zeta(e_{2})^{-1}$, so, $\zeta(e_{1}e_{2})=1$. Next,
since
\[
\tau_{1}\delta_{j}\tau_{1}=\delta_{j}^{-1},j=1,\ldots,\gamma,
\]%
\[
\tau_{1}c_{k}\tau_{1}=c_{k}^{-1},k=1,\ldots,r,
\]%
\[
\zeta(e_{1})=\zeta(e_{2})^{-1},
\]%
\[
\widehat{\Delta}^{\prime}\lhd\ker(\zeta)\;\mbox{(as $A$ is an abelian
group)},
\]
we may see that $\tau_{1}$ induces the inverse automorphism of $A$, i.e.,
\[
a\in A\mapsto a^{-1}\in A.
\]

In particular, $\tau_{1}\ker(\zeta)\tau_{1}=\ker(\zeta)$. Since,
$K=\langle\widehat{\Delta},\tau_{1}\rangle$, we obtain that $\ker(\zeta)\lhd
K$.
\end{proof}

\begin{coro}
If $S$ is a pseudo-real Riemann surface admitting an anticonformal
automorphism $\tau$ of order $2n$, then there is a real Riemann surface
$\widehat{S}$ admitting an anticonformal automorphism $\widehat{\tau}$ such
that $(\widehat{S},\widehat{\tau})$ is topologically conjugate to $(S,\tau) $.
\end{coro}

\begin{proof}
The surface $\widehat{S}$ is uniformized by $\ker(\eta)$ and we use
$\eta=\zeta$, $A=C_{2n}$, in the above lemma.
\end{proof}

Note that if $\widehat{S}$ is the surface given in the above Corollary
$D_{2n}\leq\mathrm{Aut}(\widehat{S})$.

For every pseudo-real Riemann surface $S$, Corollary 1 implies ${\mathcal{O}%
}(S,\tau)\cap{\mathcal{M}}_{g}(\mathbb{R})\neq\varnothing$ and that
$\mathrm{Fix}(\sigma)$ is connected.

%%%%%%%%%%%%%%%%%%
%%%%%%%%%%%%%%%%%%

\end{document}